\documentclass{amsart}
\usepackage{amsmath, amsthm, amssymb}
\usepackage{verbatim}

\newcommand{\vepsilon}{\varepsilon}
\newcommand{\vphi}{\varphi}

\newcommand{\cA}{\mathcal{A}}

\newcommand{\cM}{\mathcal{M}}
\newcommand{\cH}{\mathcal{H}}

\newcommand{\bC}{\mathbb{C}}

\newcommand{\supp}{\mbox{supp }}

\newtheorem{thm}{Theorem}
\newtheorem{prop}[thm]{Proposition}
\newtheorem{lem}[thm]{Lemma}
\newtheorem{cor}[thm]{Corollary}

\theoremstyle{definition}
\newtheorem{defn}[thm]{Definition}
\newtheorem{remark}[thm]{Remark}

\newtheorem{fact}[thm]{Claim}

\numberwithin{thm}{section}
\numberwithin{equation}{section}

\renewcommand{\[}{\begin{equation}}
\renewcommand{\]}{\end{equation}}

\newcommand{\wed}{\wedge}

\begin{document}

\title[H\"older continuous solutions of Monge-Amp\`ere  equations]{H\"older continuous solutions of the  Monge-Amp\`ere  equation on compact Hermitian manifolds} 
\author[S. Ko\l odziej and N.-C. Nguyen]{S\l awomir Ko\l odziej and Ngoc Cuong Nguyen} 

\address{Faculty of Mathematics and Computer Science, Jagiellonian University 30-348 Krak\'ow, \L ojasiewicza 6, Poland}
\email{Slawomir.Kolodziej@im.uj.edu.pl}

\address{Department of Mathematics and Center for Geometry and its Applications, Pohang University of Science and Technology, 37673, The Republic of Korea}
\email{cuongnn@postech.ac.kr}



\date{}
\maketitle

\begin{center}
{\em Dedicated to Jean-Pierre Demailly  on the occasion of his 60th birthday}
\end{center}

\bigskip

\begin{abstract}
We show that a positive Borel measure of positive finite total mass, on compact Hermitian manifolds, admits a H\"older continuous quasi-plurisubharmonic solution to the Monge-Amp\`ere equation if and only if it is dominated locally by Monge-Amp\`ere measures of H\"older continuous plurisubharmonic functions.
\end{abstract}

\section{Introduction}

The analogue of the Calabi-Yau theorem  on compact Hermitian manifolds was proven in 2010  by Tosatti and Weinkove   \cite{TW10b}. 
Continuous weak solutions for the right hand side in $L^p , \ p>1$ were obtained later by the authors \cite{KN15}. Here we continue to study
weak solutions for more general measures.

Consider   a compact Hermitian manifold $(X,\omega)$ of dimension $n$, and  a positive Radon measure $\mu$ with finite total mass  on $X$.
An upper semicontinuous  function $u$ on $X$ is called $\omega-$psh if $dd^c u +\omega \geq 0 $ (as currents). Then we write $u \in PSH(\omega) .$
Our objective is to show that if the complex Monge-Amp\`ere equation has H\"older continuous solutions for $\mu$ restricted to local charts then
it has H\"older continuous solutions globally on $X$. To be precise we introduce first  the following definition.

\begin{defn}
\label{defn:holder-sub}  We say that  $\mu$ admits a global H\"older continuous {\em subsolution}  if there exists a H\"older continuous $\omega-$psh function $u$ and $C_0>0$ such that
\[\label{eq:holder-sub}
\mu \leq C_0 (\omega + dd^c u)^n \quad \mbox{ on } X.
\]
Let us denote by $\cM$  the set of all such measures.
\end{defn}

To verify the defining condition it is enough to look at $\mu $  locally.

\begin{lem}\label{lem:global-local-sub-solution} A measure $\mu$ belongs to $\cM$ if and only if   for every $x\in X$, there exists a neighborhood $D$ of $x$ and a H\"older continuous psh function $v$ 
on $D$ such that $\mu_{|_D} \leq (dd^c v)^n$. 
\end{lem}

\begin{proof}  The necessary condition is obvious, so we   prove the sufficient condition. Using the strict positivity of $\omega$ we can extend a H\"older continuous psh function $v$  defined in a local coordinate chart to the whole space $X$ so that the extension is  a H\"older continuous $C \omega-$psh function for some large $C>0$. Taking a finite cover by coordinate charts and using the partition of unity one easily constructs  a global $\omega-$psh  function $u$ satisfying \eqref{eq:holder-sub} (see \cite{kol05} for details of such a construction).
\end{proof}

Our main result can be viewed as a generalization of  Demailly et al.  \cite[Proposition~4.3]{demailly-et-al14} from the K\"ahler to the Hermitian setting.

\begin{thm}\label{thm:holder}  Assume that $0< \mu(X) < +\infty$. There exists a H\"older continuous $\omega$-psh $\vphi$ and a constant $c>0$ solving $$(\omega + dd^c \vphi)^n  = c\; \mu$$
if and only if  $\mu$ belongs to $\cM$. 
\end{thm}

Thanks to this theorem the important class of measures having $L^p$-density, for $p>1,$  admits  H\"older continuous solutions. 
This result was proven in \cite[Theorem~B]{KN2} under the extra assumption that the right hand side is strictly positive.

\begin{cor} \label{cor:lp} Let $f$ be a non-negative function in $ L^p(\omega^n)$ for $p>1$. Assume that  $\int_X f \omega^n >0$. Then there exists a H\"older continuous $\vphi \in PSH(\omega)$ and a constant $c>0$ such that
$$(\omega + dd^c \vphi)^n = c f\omega^n.$$
\end{cor}

\begin{proof} By \cite[Theorem~0.1]{KN15} there exists $\vphi \in PSH(\omega) \cap C^0(X)$ and a constant $c>0$ satisfying 
\[\notag
	(\omega+ dd^c\vphi)^n = c f\omega^n.
\]
Consider  a local coordinate chart $B \subset \subset X$  parametrized by
 the unit ball in $\bC^n$. Let $\chi$ be a smooth cut-off function such that
\[\notag
	0\leq \chi \leq 1, \quad \chi = 1 \mbox{ on } B(0, 1/2), \quad
	\supp \chi \subset\subset B.
\]
Find $w \in PSH(B)$ the solution of the Dirichlet problem for the Monge-Amp\`ere equation:
\[\notag
	(dd^c w)^n = c \chi f \omega^n, \quad w_{|_{\partial B}} =0.
\]
By the main result of \cite{GKZ08} (see also \cite{Cha15a}) we get that $w \in C^{0,\alpha}(\bar{B})$ for some $\alpha$ positive depending only on $n,p.$ Therefore, on $B(0,1/2)$ the right hand  side $cf \omega^n$ is dominated by $(dd^cw)^n$. We conclude from Lemma~\ref{lem:global-local-sub-solution} and Theorem~\ref{thm:holder} that  $\vphi$ is H\"older continuous.
\end{proof}

\begin{remark} Using the recent result from \cite{cuong17} instead of \cite{GKZ08} we also can show that if $\mu \in \cM$ and $0\leq f \in L^p(d\mu)$ for $p>1$, then $f d\mu \in \cM$.  In other words, $\cM$ satisfies the $L^p-$property (see \cite{demailly-et-al14}) and the above corollary is a special case.
\end{remark}

Another consequence of the main result is the  convexity of the range of Monge-Amp\`ere operator acting on H\"older continuous functions.

\begin{cor} \label{cor:convexity} The set 
$$\cA := \left\{  c \cdot (\omega + dd^c \vphi)^n: \begin{aligned}  \vphi \in PSH(\omega),\; \vphi \mbox{ is H\"older continuous, } c>0.
\end{aligned}
\right\}
$$
is  convex.
\end{cor}

\begin{proof} For brevity we use the notation $\omega_{\vphi}^n := (\omega + dd^c \vphi)^n$.  Let $c_1\omega_{\vphi_1}^n, c_2 \omega_{\vphi_2}^n \in \cA.$ It is easy to see that
\[\notag
	\mu:= \frac{1}{2} (c_1 \omega_{\vphi_1}^n + c_2\omega_{\varphi_2}^n) \leq  2^{n-1} (c_1+ c_2) \left(\omega + dd^c \frac{\vphi_1 + \vphi_2}{2}\right)^n.
\]
Apply Theorem~\ref{thm:holder} to get that 
$
	\omega_\phi^n = c \mu
$
for some H\"older continuous $\omega$-psh $\phi$ and some constant $c>0.$ Therefore, $\mu \in \cA.$
\end{proof}

\bigskip

\bigskip
{\bf Dedication. }  It is our privilege to dedicate this paper to Jean-Pierre Demailly, a great mathematician and a champion for math education.

\bigskip

{\bf Acknowledgement.} The first author was partially supported by NCN grant  
2013/ 08/A/ST1/00312. The second author was supported by  the NRF Grant 2011-0030044 (SRC-GAIA) of The Republic of Korea. He also would like to thank Kang-Tae Kim for encouragement and support.

\section{Preliminaries}

Let us recall the definition of the Bedford-Taylor capacity. For a Borel set  $E\subset X$ put
\[
	cap_\omega(E) := \sup \left\{\int_E \omega_v^n : v\in PSH(\omega), 0\leq v \leq 1\right\}.
\] 
By  \cite[p. 52]{kol05},  this capacity is comparable with the local Beford-Taylor capacity $cap_\omega'(E)$. 
Combining this fact  with the work of Dinh-Nguyen-Sibony \cite{DNS10} we get the following result.

\begin{lem} \label{lem:DNS}  Let $\mu \in \cM$. Then,  for every compact set $K\subset X$,
\[\label{eq:vol-cap}
	\mu(K) \leq C \exp \left( \frac{-\alpha_1}{[cap_\omega(K)]^\frac{1}{n}} \right),
\]
where $C, \alpha_1>0$ depend only on $X$ and the  H\"older exponent of the global H\"older subsolution.\end{lem}

\begin{cor}\label{cor:halpha} Assume that $\mu \in \cM$ and fix $\tau>0$. Then, there exists $C_\tau >0$ such that for every compact set $K \subset X$
\[\label{eq:htau}
	\mu(K) \leq C_\tau \left[cap_\omega(K)\right]^{1+\tau}.
\]
The set of measures satisfying this inequality is denoted by $ \cH(\tau)$.
\end{cor}

The proof of the next statement can be found in  \cite[Theorem~2.1]{DMN15}.
\begin{lem}\label{lem:smooth-approximation} Let $u\in PSH(\omega)  \cap C^{0,\alpha} (X)$ with $0<\alpha<1$. Then there exists a sequence of smooth $\omega$-psh function $\{u_j\}_{j\geq 1}$ such that 
$$u_j \rightarrow u $$
in $C^{0,\alpha'}(X)$  as $j \to +\infty$,  for any $0< \alpha' < \alpha$.
\end{lem}

We need also an estimate which for K\"ahler manifolds was given in \cite{EGZ09}.

\begin{prop}\label{prop:l1-stability} Suppose $\psi \in PSH(\omega) \cap C^0(X)$ and $\psi \leq 0$. Let $\mu$ satisfy the inequality \eqref{eq:htau} for some $\tau>0$, i.e.  $\mu \in \cH(\tau)$. Assume that $\vphi \in PSH(\omega) \cap C^0(X)$ solves
\[\notag
	(\omega +dd^c \vphi)^n = \mu .
\]
 Then for  $\gamma = \frac{1}{1+ (n+2)(n+\frac{1}{\tau }) }$ and some positive $C>0$  depending only on $\tau, \omega$ and $\|\psi\|_\infty$ we have
\[\notag
	\sup_X(\psi - \vphi) \leq C \left\|(\psi - \vphi)_+\right\|_{L^1(d\mu)}^{\gamma}.
\]

\end{prop}

\begin{proof} Without loss of generality we may assume that $-1\leq \psi \leq 0$. Put \[ \notag U(\vepsilon, s) = \{\vphi<(1-\vepsilon) \psi + \inf_X [\vphi -(1-\vepsilon) \psi] +s \},\]
where $0<\vepsilon<1$ and $s>0.$

\begin{lem} 
\label{lem:cap-level-set}
For  
$0 <s \leq \frac{1}{3}\min\{\varepsilon^n, \frac{\varepsilon^3}{16 B} \}$, 
$0< t \leq \frac{4}{3} (1-\varepsilon) \min\{\varepsilon^n, \frac{\varepsilon^3}{16 B} \}$ we have
\[\notag
	t^n \, cap_{\omega} (U(\varepsilon, s))
	\leq 	 C \left[cap_\omega (U(\vepsilon, s+t))\right]^{1+\tau},
\]
where $C$ is a dimensional constant.
\end{lem}

\begin{proof}[Proof of Lemma~\ref{lem:cap-level-set}]
By \cite[Lemma~5.4]{KN15} 
\[\label{eq:cap-growth}
	t^n \, cap_{\omega} (U(\varepsilon, s))
	\leq 	 C \,  \int_{ U (\varepsilon, s + t) }	\omega_\varphi^n,
\]
The lemma now follows     from \eqref{eq:htau}.                             
\end{proof}


\begin{lem}\label{lem:uniform-apriori-estimate}
Fix $0<\vepsilon<3/4$ and $\vepsilon_B := \frac{1}{3} \min\{\vepsilon^n, \frac{\vepsilon^3}{16B}\}$. Then, there exists a contant $C_\tau = C(\tau, \omega)$ such that for $0<s<\vepsilon_B$, 
\[\notag
	s\leq C_\tau 
	\left[cap_\omega(U(\vepsilon,s)) \right]^\frac{\tau}{n}.
\]
\end{lem}

\begin{proof}[Proof of Lemma~\ref{lem:uniform-apriori-estimate}] 
Let us use the notation 
$$ a(s) := \left[cap_\omega (U(\vepsilon, s)) \right]^\frac{1}{n}.
$$
It follows easily from \eqref{eq:cap-growth} that 
\[\notag
	t a(s) \leq C \left[a(s+t)\right]^{1+\tau}.
\]
This is the inequality \cite[Eq. (3.6)]{KN3}. The arguments that follow in that paper complete the proof of the present lemma.
\end{proof}

To finish the proof of the proposition  we proceed as in \cite[Theorem~3.11]{KN3}.  One needs  to estimate
$$-S:= \sup_X (\psi - \vphi) > 0$$ in terms of $\|(\psi - \vphi)_+\|_{L^1(d\mu)}$ as in the K\"ahler case \cite{kol03}. Suppose that 
\[ \label{slr-eq2} \|(\psi - \vphi)_+\|_{L^1(d\mu)} \leq \vepsilon^{a} 
\]
for $0< \vepsilon << 3/4$ and $a=\frac{1}{\gamma }$.
Let  
\[ \notag \hbar(s) := (s/C_\tau)^\frac{1}{\tau} \]  be the inverse function of $C_\tau   s^\tau$. 
 Consider sublevel sets $U(\vepsilon, t) = \{\vphi< (1-\vepsilon) \psi + S_\vepsilon +t \}$, where $S_\vepsilon = \inf_X [\vphi -(1-\vepsilon)\psi]$. 
It is clear that \[ S - \vepsilon \leq S_\vepsilon \leq S.\] 
Therefore, $U(\vepsilon,2t) \subset \{\vphi < \psi + S+ \vepsilon +2t\}$. Then, $(\psi - \vphi)_+ \geq |S| - \vepsilon -2t>0$ for $0< t < \vepsilon_B$ and $0< \vepsilon < |S|/2$ 
on the latter set (if $|S| \leq 2 \vepsilon$ then we are done).

By \eqref{eq:cap-growth} we have
\begin{align*}
	cap_{\omega}(U(\vepsilon,t)) 
	\leq \frac{C}{t^n} \int_{U(\vepsilon,2t)} d\mu
&	\leq \frac{C}{t^n} \int_X \frac{(\psi -\vphi)_+}{(|S| - \vepsilon -2t)}
		d\mu\\
&	\leq \frac{C \|(\psi - \vphi)_+\|_{L^1(d\mu)}}{t^n (|S| - \vepsilon -2t)} .
\end{align*}
Moreover, by Lemma~\ref{lem:uniform-apriori-estimate} \[\notag \hbar(t) \leq[ cap_\omega(U(\vepsilon,t))]^\frac{1}{n}. \] 
Combining these inequalites, we obtain
\[\notag
	(|S| - \vepsilon - 2t) 
	\leq \frac{C \|(\psi - \vphi)_+\|_{L^1(d\mu)}}{t^n [\hbar(t)]^n} .
\]
Therefore, using \eqref{slr-eq2},
\begin{align*}
|S| 
&	\leq \vepsilon + 2t 
	+ \frac{C \|(\psi - \vphi)_+\|_{L^1(d\mu)}}{t^n [\hbar(t)]^n}  \\
&	\leq 3 \vepsilon + \frac{C \vepsilon^a}{t^n [\hbar(t)]^n}.
\end{align*}
Recall that $\vepsilon_B = \frac{1}{3} \min\{\vepsilon^n, \frac{\vepsilon^3}{16B}\}$. So, taking 
$
	t = \vepsilon_B/2 \geq \vepsilon^{n+2} $
we have
\[\notag
	\hbar(t) = \left(\frac{t}{C_\tau }\right)^{1/\tau} 
	\geq C \vepsilon^{(n+2)/\tau}.
\]
With our choice of $a$ 
\[\notag \frac{\vepsilon^{a}}{\vepsilon^{n(n+2)+ \frac{(n+2)}{\tau}}} =\vepsilon.\] 
Hence  $|S| \leq C \vepsilon$ with $C = C(\tau, \omega)$. Thus,
\[\notag
	\sup_X(\psi -\vphi) \leq C \|(\psi - \vphi)_+\|_{L^1(d\mu)}^\frac{1}{a}.
\]
This is the desired stability estimate.
\end{proof}

Following \cite{De94} we 
consider $\rho_{\delta }\vphi$- the regularization of the $\omega$-psh function $\vphi$ defined  by

\begin{equation}\label{eq:phie}
\rho_\delta \vphi(z)=\frac{1}{\delta ^{2n}}\int_{\zeta\in T_{z}X}
\vphi({\exp} h_z(\zeta))\rho\Big(\frac{|\zeta|^2_{\omega }}{\delta ^2}\Big)\,dV_{\omega}(\zeta),\ \delta>0;
\end{equation}
where $\zeta \to {\exp}h_z(\zeta)$ is the (formal) holomorphic part of the Taylor expansion of the exponential map of  the Chern connection on the tangent bundle of $X$ associated to $\omega $, and the modifier 
 $\rho: \mathbb R_{+}\rightarrow\mathbb R_{+}$  is given by
$$\rho(t)=\begin{cases}\frac {\eta}{(1-t)^2}\exp(\frac 1{t-1})&\ {\rm if}\ 0\leq t\leq 1,\\0&\
{\rm if}\ t>1\end{cases}$$
 with the constant $\eta$ chosen so that
\begin{equation}\label{total integral}
\int_{\mathbb C^n}\rho(\Vert z\Vert^2)\,dV(z)=1 ,
\end{equation}
where $dV$ denotes the Lebesgue measure in $\mathbb C^n$).

The proof of the following variation of  \cite[Proposition~3.8]{De94} and \cite[Lemma 1.12]{BD12} was given in \cite{KN2}.

\begin{lem}\label{kis}
 Fix $\vphi \in PSH(\omega) \cap L^\infty(X)$.  Define the Kiselman-Legendre transform with level $b>0$ by
 \begin{equation}\label{kisleg}
 \Phi_{\delta , b } (z)= \inf _{ t\in [0,\delta ]}\left(\rho_{t }\vphi (z)+ Kt^2  + Kt -b \log\frac{t}{\delta } \right),
 \end{equation}
Then for  some positive constant $K$  depending on the curvature, the function $\rho_{t } \vphi+Kt^2$ is increasing in $t$ and
 the following estimate holds:
\begin{equation}\label{hessest}
\omega+dd^c  \Phi_{\delta,b }\geq -(A b+2K\delta)\,\omega,
\end{equation}
where $A$ is a lower bound of the negative part of the Chern curvature of $\omega$.
\end{lem}

The next lemma is essentially proven in  \cite[Theorem 4.3]{demailly-et-al14} or \linebreak 
\cite[Lemma~3.3,  Proposition~4.4]{DN16}.
The adaption of those proofs to the case of  compact Hermitian manifolds is straightforward.

\begin{lem}\label{lem:l1-bound} Let $\mu \in \cM$ and  $\vphi \in PSH(\omega) \cap L^\infty(X)$. Then, there exists $0< \alpha_1 < 1$ such that
\[\label{eq:l1-differ-regularisation-phi}
	\|\rho_\delta \vphi - \vphi\|_{L^1(d\mu)} \leq C \delta^{\alpha_1}. 
\]
\end{lem}

\section{Proof of Theorem~\ref{thm:holder}}

The necessary condition follows easily. It remains to prove the other one. As $\mu \in \cM$ there exists $u \in PSH(\omega) \cap C^{0,\alpha _0}(X)$ with $0< \alpha _0 \leq 1$, and $C_0>0$ such that
\[
	\mu \leq C_0 (\omega + dd^c u)^n.
\]
Using Radon-Nikodym's theorem, we write $\mu = C_0 h \omega_u^n$ for a Borel measurable function $0\leq h \leq 1$. Let $u_j$ be the smooth approximation of $u$ as in Lemma~\ref{lem:smooth-approximation} and denote $$\mu_j := C_0 h \omega_{u_j}^n.$$ Then $\mu_j$ converges weakly to $\mu$ as $j\to +\infty$. Using \cite[Theorem~0.1]{KN15} we find $\vphi_j \in PSH(\omega) \cap C^0(X)$ with normalisation $\sup_X \vphi_j =0$, and $c_j>0$ satisfying
\[
	\omega_{\vphi_j}^n = c_j \mu_j. 
\]
The first thing we need to show is the following.

\begin{fact}\label{fact:bound-cst} 
There is a uniform constant $C_1>0$ such that $1/C_1 < c_j < C_1.$
\end{fact}

\begin{proof} Since $\mu(X) >0$, it follows that $\int_X h \omega_u^n >0.$ Therefore, $\int_X h^\frac{1}{n} \omega_u^n >0.$ By the 
Bedford-Taylor convergence theorem \cite{BT82}  we know that $\omega_{u_j}^n$ converges weakly to $\omega_u^n$. Since $C^0(X)$ is dense in $L^1(X, \omega_u^n)$, we have
\[\notag
	\int_X h^\frac{1}{n} \omega_{u_j}^n > C
\]
for some uniform $C>0$.  Applying the mixed forms type inequality (see \cite{kol05}, \cite{cuong16}) one obtains
\[\notag
	\omega_{\vphi_j} \wed \omega_{u_j}^{n-1} \geq \left[ \frac{\omega_{\vphi_j}^n}{\omega_{u_j}^n}\right]^\frac{1}{n} \omega_{u_j}^n = (c_jC_0 h)^\frac{1}{n} \omega_{u_j}^n.
\]
On the other hand,
\[\notag
\begin{aligned}
\int_X \omega_{\varphi_j} \wed \omega_{u_j}^{n-1} 
&=	\int_X \omega \wed \omega_{u_j}^{n-1} + \int_X dd^c \vphi_j \wed \omega_{u_j}^{n-1} \\
&= 	\int_X \omega \wed \omega_{u_j}^{n-1} + \int_X \vphi_j dd^c (\omega_{u_j}^{n-1})\\
&\leq		\int_X \omega \wed \omega_{u_j}^{n-1} + B \int_X |\vphi_j| (\omega^2 \wed \omega_{u_j}^{n-2} + \omega^3\wed \omega_{u_j}^{n-3}),
\end{aligned}\]
where $B$ is a constant depending only on $\omega $ (see e.g. \cite{DK12} for details).
Since $\|u_j\|_\infty < C$ and $\sup_X\vphi_j =0$, it follows from the Chern-Levine-Nirenberg type inequality (\cite[Proposition~1.1]{cuong16}) that the right hand side is uniformly bounded. Thus,
\[\notag
	\int_X \omega_{\vphi_j} \wed \omega_{u_j}^{n-1} \leq C.
\]
Combining the above inequalities we get 
 $$c_j < C_1:= \frac{\int_X \omega_{\vphi_j} \wed \omega_{u_j}^{n-1}}{\int_X (C_0 h)^\frac{1}{n} \omega_{u_j}^n} < +\infty.$$
Hence, by \cite[Lemma~5.9]{KN15} we also have  $$c_j > 1/C_1,$$
increasing $C_1$ if necessary. Thus, Claim~\ref{fact:bound-cst} is proven.
\end{proof}

Thanks to  Lemma~\ref{lem:DNS}, Lemma\ \ref{lem:smooth-approximation} and Claim~\ref{fact:bound-cst} measures $\mu _j$  satisfy the volume-capacity inequality \eqref{eq:vol-cap}
with a uniform constant. 
Thus by \cite[Corollary~5.6]{KN15} we have $\|\vphi_j\|_\infty < C_2.$ Passing to a subsequence one may assume  that $\{\vphi_j\}$ is a Cauchy sequence in $L^1(\omega^n)$, and 
$\{c_j\}$ converges. Set
\[\label{eq:limit-eq}
	\vphi := (\limsup_j \vphi_j)^*, \quad c = \lim_j c_j.
\] 
Again passing to a subsequence if necessary we can also assume that
\[
	\vphi_j \rightarrow \vphi \quad \mbox{in } L^1(\omega^n) \quad
	\mbox{as } j \to \infty.
\]

\begin{lem} \label{lem:capacity-convergence} We have $$\int_X |\vphi_k - \vphi| \omega_{u_j}^n \to 0 \quad\mbox{ as } \min\{j, k\} \to \infty.$$
\end{lem}

\begin{proof}
Using the uniform boundedness of $\|\vphi_j\|_\infty$, $\|u_j\|_\infty$ and the argument in Cegrell\cite[Lemma~5.2]{Ce98} (it's a version of Vitali's convergence theorem) we get that
\[ \label{vitali-thm}
\int_X |\vphi_k - \vphi| \omega_{u}^n \to 0 \quad \mbox{ as }  k \to \infty.
\]
Indeed, we first have $\int_X (\vphi_k - \vphi) \omega_{u}^n \to 0$ as $k \to \infty$. Moreover, all functions are negative we get the result.

We shall prove the lemma by the contradiction argument. Assume that there exist subsequences, still denoted by $\{\vphi_k\}_{k\geq 1}^\infty$, $\{u_j\}_{j\geq 1}^\infty$, and $\delta>0$ such that
\[\notag
	\int_X |\vphi_k - \vphi| \omega_{u_j}^n > \delta.
\]
Let $a>0$ be small. By Hartogs'  lemma there exists $k_0$ such that
\[\notag
	\vphi_k \leq \vphi +a \quad \forall k \geq k_0.
\]
If we choose $a$ small enough, then for $k\geq k_0$ and $j\geq 1$,
\[\label{eq:lower-bound}
	\int_X (\vphi - \vphi_k) \omega_{u_j}^n \geq \delta/2.
\]
Next, we are going to show that
\[\label{eq:difference}
	E_{jk}:= \int_X (\vphi - \vphi_k) \omega_{u_j}^n - \int_{X} (\vphi - \vphi_k) \omega_u^n \to 0
\]
as $\min\{j,k \}\to +\infty.$ Indeed, 
\[\notag
	E_{jk} = \int_X (\vphi - \vphi_k) dd^c (u_j - u) \wed \sum_{p=0}^{n-1} \omega_{u_j}^p \wed \omega_u^{n-1-p}.
\]
Let us denote by $T_p (j)$ the current $ \omega_{u_j}^p \wed \omega_u^{n-p-1}$. Then
\[\notag
\begin{aligned}
	dd^c \left[(\vphi - \vphi_k) T_p (j)\right] 
&	= dd^c (\vphi -\vphi_k) \wed T_p (j)+ d(\vphi-\vphi_k) \wed d^c T_p (j)\\
&\quad - d^c (\vphi-\vphi_k) \wed dT_p (j) + (\vphi-\vphi_k) dd^c T_p (j)\\
&\quad =: S_1 + S_2 + S_3 + S_4.
\end{aligned}\]
By integration by parts 
\[\begin{aligned}
	E_{jk} 
&= 		\int_X (u_j - u) dd^c\left[ (\vphi - \vphi_k) T_p (j)\right] \\
&= 		\int_X (u_j -u) (S_1+S_2+S_3+S_4).
\end{aligned}\]
Now we shall estimate each term in the right hand side. First, since $S_1 = (\omega_\vphi - \omega_{\vphi_k}) \wed T_p (j)$,
\[
	\left|\int_X (u- u_j) S_1\right| \leq \|u - u_j\|_{\infty} \left(\int_X (\omega_\vphi + \omega_{\vphi_k}) \wed T_p (j)\right) \to 0
\]
as $j\to +\infty.$

Next, we estimate $\int_X (u-u_j) S_2$. As $d^c T_p (j)=  d^c \omega \wed T_p' (j) ,$ where $T_p' (j)$ is a sum of terms  of the form $C_3  \omega_{u_j}^{k}\wed \omega_u^{q}$ (the constant $C_3$ depending only on $n,p$), we apply the Cauchy-Schwarz inequality \cite[Proposition~1.4]{cuong16} to get that
\[\begin{aligned}
&\left| \int_X (u-u_j) d\vphi \wed S_2 \right| \\
&\leq C \|u-u_j\|_\infty \left[ \int_X d\vphi \wed d^c \vphi \wed \omega \wed T_p' (j) \right]^\frac{1}{2} \left[\int_X \omega^2\wed T_p' (j)\right]^\frac{1}{2}.
\end{aligned}\]
Moreover, 
\[\begin{aligned}
	2 \int_X d\vphi \wed d^c \vphi\wed T_p' (j)
&=	\int_X dd^c \vphi^2 \wed T_p' (j)- \int_X 2\vphi \omega_\vphi \wed T_p' (j) \\
&\quad+ 2 \int_X \omega\wed T_p' (j) \\
&\leq		C\left(\int_X \omega^n + \|\vphi\|_\infty^n \|u_j\|_\infty^n \|u\|_\infty^n\right),
\end{aligned}\]
where in the last inequality we used \cite[Proposition~1.5]{cuong16}.
Therefore, we conclude the right hand side of the previous inequality tends to $0$ as $j\to +\infty$. 
Similar estimates are also applied to the remaining terms with $S_3, S_4$. Thus we have shown that $E_{jk} \to 0$ as $\min\{j,k\} \to +\infty.$

Combining \eqref{vitali-thm}, \eqref{eq:lower-bound}, and \eqref{eq:difference} we get a contradiction. The lemma thus follows.
\end{proof}

{\em Existence of a continuous solution.} Notice that
\[\notag
	\int_X |\vphi_j - \vphi_k| \omega_{u_j}^n \leq \int_X |\vphi_j -\vphi| \omega_{u_j}^n + \int_X |\vphi_k -\vphi| \omega_{u_j}^n \to 0
\]
as $\min\{j,k \} \to +\infty$.  Therefore, using Lemma~\ref{lem:capacity-convergence} and the argument in \cite[Theorem~5.8]{KN15} we  get that
$\{\vphi_j\}_{j\geq 1}$ is a Cauchy sequence in $C^0(X)$. Thus, $$\vphi= \lim_j \vphi_j \quad \mbox{ in } C^0(X).$$ 
We conclude that  $\vphi \in PSH(\omega) \cap C^0(X)$ and it solves 
\[\label{eq:cont-solution}
	\omega_{\vphi}^n = c \; \mu,
\]
where $c$ is defined in \eqref{eq:limit-eq}.

{\em H\"older continuity of the solution.} We shall show that the solution $\vphi$ obtained in \eqref{eq:cont-solution} is H\"older continuous. 
Fix $\tau >0$ and set
$$
\alpha = \min \left\{\frac{1}{1+ (n+2)(n+\frac{1}{\tau }) } , \alpha _1\right\},
$$
where $\alpha _1$ is given in Lemma~\ref{lem:l1-bound}. By Corollary~\ref{cor:halpha} $\mu\in \cH(\tau)$ and then 
Proposition~\ref{prop:l1-stability} holds with $\gamma =\alpha .$

Consider the regularization of $\vphi$ as in \eqref{eq:phie}. 
As explained in \cite{kol08} and \cite{demailly-et-al14}   the result follows as soon as we show that
$$
\rho _t \vphi - \vphi \leq C t^{\alpha \alpha_1}
$$
for $t$ small enough.

It follows from Lemma~\ref{kis} that 
\[\notag
\begin{aligned}
	\vphi \leq \Phi_{\delta,b} 
&\leq \rho_\delta \vphi+ K(\delta+\delta^2). \\
&\leq \rho_\delta\vphi + 2K\delta.	
\end{aligned}\]
Choose the level $b = \left(\delta^\alpha -2 K\delta\right)/A = O(\delta^\alpha)$ so that \[Ab + 2K \delta = \delta^\alpha.\] 
After fixing the level $b$, we write 
\[\label{eq:modify} \Phi_\delta:= (1- \delta^\alpha)\Phi_{\delta,b}.\]
Then, by Lemma~\ref{kis}
\[
	\omega + dd^c \Phi_\delta \geq \delta^{2\alpha} \omega.
\]
Since $- C_4 \leq \vphi \leq 0$ and $\rho_\delta \vphi \leq 0$ one obtains
\[\label{eq:diff-0}
	\Phi_\delta 
	\leq (1-\delta^\alpha) (\rho_\delta \vphi+ K\delta + K\delta^2) \leq 2 K\delta.
\] It follows that
\[\label{eq:diff-1}
	\Phi_\delta \leq C_4 \delta^\alpha
\]
for $\delta \leq \delta_0$ small. Therefore, by \eqref{eq:diff-0} and \eqref{eq:diff-1} we have
\[\label{eq:diff-2}
	\Phi_\delta - \vphi \leq C_4\delta^\alpha + (1-\delta^\alpha) (\rho_\delta \vphi + K\delta + K\delta^2 -\vphi).
\]
Next, the stability estimate Proposition~\ref{prop:l1-stability}  applied for $\Phi_\delta - C_4 \delta^\alpha$ and $\vphi$, and $\gamma = \alpha$ give us  that
\[ \notag
\begin{aligned}
	\sup_X(\Phi_\delta - \vphi)
&	\leq C_5 \| \max\{ \Phi_\delta - \vphi - C_4 \delta^\alpha,0\}\|_{L^1(d\mu)}^\alpha + C_4 \delta^\alpha\\
&	\leq C_5 \| \rho_\delta \vphi+ K\delta + K\delta^2 -\vphi\|_{L^1(d\mu)}^\alpha +  C_4 \delta^\alpha,
\end{aligned}
\]
where we used \eqref{eq:diff-2} for the second inequality.
Hence, using Lemma~\ref{lem:l1-bound}, we conclude that 
\[\label{eq:diff-3}
	\Phi_\delta - \vphi \leq C_6 \delta^{\alpha\alpha_1}.
\]

For a fixed point $z$,  the minimum in the definition of $\Phi_{\delta,b}(z)$  is realized for  some   $t_0 = t_0 (z)$. Then, \eqref{eq:modify} and \eqref{eq:diff-1} imply 
\[\notag
	(1- \delta^\alpha) (\rho_{t_0} \vphi + Kt_0 + Kt_0 ^2  - b \log \frac{t_0 }{\delta} - \vphi) \leq C_6 \delta^\alpha.
\]
Since $\rho_t \vphi + Kt^2 + Kt - \vphi \geq 0$, we have
\[\notag 
	b (1-\delta^\alpha) \log\frac{t_0}{\delta} \geq - C_6 \delta^\alpha.
\]
Combining this with $b \geq \delta^\alpha / (2A)$, one gets that
\[
	t_0(z) \geq \delta \kappa \quad \mbox{ for } \kappa = \exp \left(- \frac{2AC_6}{(1-\delta_0^\alpha)}\right),
\]
where $\delta_0$ is fixed, and $\kappa$  is a  uniform constant.

Now, we are ready to conclude the proof. Since $t_0 =t_0 (z)\geq \delta \kappa$ and $t \mapsto \rho_t \vphi + Kt^2$ is increasing, 
\[\begin{aligned} \notag
	\rho_{\kappa \delta} \vphi(z)  + K (\delta \kappa)^2 + K\delta \kappa - \vphi(z) 
&	\leq \rho_{t_0} \vphi (z) + Kt_0^2 + K t_0 -\vphi (z)  \\
&	= \Phi_{\delta,b} (z) - \vphi (z) \\
&	= \frac{\delta^\alpha}{1-\delta^\alpha} \Phi_\delta  + (\Phi_\delta -\vphi). 
\end{aligned}\]
Combining this, \eqref{eq:diff-1} and \eqref{eq:diff-3} we get that
$$
	\rho_{\kappa \delta} \vphi (z) - \vphi(z) \leq C_7 \delta^{\alpha\alpha_1}. 
$$
The desired estimate  follows by rescaling $\delta := \kappa \delta$ and increasing $C_7$.

\end{document}